\documentclass[12pt]{article}

\usepackage{amsmath}
\usepackage{amsfonts}
\usepackage{amsthm}
\usepackage{mathrsfs}
\usepackage{graphicx}
\usepackage{framed}
\usepackage{enumitem}
\usepackage[pagewise, displaymath, mathlines]{lineno}
\usepackage{fancyhdr,lipsum}
\usepackage[colorlinks = true,
            linkcolor = blue,
            urlcolor  = blue,
            citecolor = blue,
            anchorcolor = blue]{hyperref}

\newtheorem{theorem}{Theorem}[section]
\newtheorem{lemma}[theorem]{Lemma}

\newtheorem{proposition}[theorem]{Proposition}
\newtheorem{definition}[theorem]{Definition}

\theoremstyle{remark}
\newtheorem{remark}{Remark}[section]

\theoremstyle{definition}
\newtheorem{Q}{Question}[section]

\def\<{\langle}
\def\>{\rangle}

\newcommand{\romd}{\mathrm{d}}

\newcommand{\lspan}{\mathrm{span}}

\newcommand{\aaa}{\alpha}
\newcommand{\bb}{\beta}

\newcommand{\slog}{\mathrm{slog}}

\begin{document}

\title{Almost bi--Lipschitz embeddings using covers of balls centred at the origin}

\author{Alexandros Margaris}
\date{}

\maketitle

\begin{abstract}
In 2010, Olson \& Robinson [Transactions of the American Mathematical Society, 362(1), 145-168] introduced the notion of an almost homogeneous metric space and showed that if $X$ is a subset of a Hilbert space such that $X-X$ is almost homogeneous, then $X$ admits almost bi--Lipschitz embeddings into Euclidean spaces. In this paper, we extend this result and we show that if $X$ is a subset of a Banach space such that $X-X$ is almost homogeneous at the origin, then $X$ can be embedded in a Euclidean space in an almost bi--Lipschitz way.
\end{abstract}

\section{Introduction}

We say that a metric space $(X,d)$ is doubling if every ball in $X$ can be covered by a fixed number of balls in $X$ with half the radius. In 1983, Assouad \cite{Ass} proved that for any $0< \epsilon<1$, doubling metric spaces admit bi--H\"older embeddings with H\"older exponent $\epsilon$, into an Euclidean space $\mathbb{R}^{N}$, for some $N$ depending on $\epsilon$. Assouad's pioneering work triggered a lot of research in the field of embeddings. Recent results by Naor \& Neiman \cite{NN}, based on work by Abraham, Bartal \& Neiman \cite{ABN}, showed that we can actually choose $N$ in Assouad's theorem to be independent of $\epsilon$. The same result was also achieved independently by David \& Snipes \cite{snipes}.

The doubling property remains invariant under bi--Lpschitz maps  and we also know that any subset of an Euclidean space is doubling (see Chapter 9 in the book of Robinson \cite{JCR}). Hence, a metric space needs to be doubling in order to admit a bi--Lipschitz embedding into an Euclidean space. However, the condition is not sufficient since there are examples of doubling spaces due to Lang \& Plaut \cite{LP}, Semmes \cite{SE} and Pansu \cite{PANSU} that cannot be embedded in a bi--Lipschitz way into any Hilbert space.

An equivalent definition of a doubling space is given which involves the Assouad dimension $\romd_{A}(X)$. The Assouad dimension is related to homogeneous sets, which are defined as follows.

\begin{definition}\label{HOM}
A subset $V$ of a metric space $(X,d)$ is said to be $(M,s)-homogeneous$ if for every $x \in V$ and $r > \rho > 0$
$$N_{V}(r,\rho) = N(V \cap B(x,r), \rho) \leq M\left(\frac{R}{r}\right)^{s},$$
where $N(V \cap B(x,r), \rho)$ denotes the minimum number of balls of radius $\rho$ required to cover $V \cap B(x,r)$. 

The Assouad dimension of $V \subset (X,d)$, $\romd_{A}(V)$ is defined as the infimum of all $s >0$ such that $V$ is $(M,s)$ homogeneous for some $M >0$.
\end{definition}

Another notion of dimension which will be used later is the box--counting dimension whose definition we now recall.

\begin{definition}\label{BC2}
Suppose that $(E,\|\cdot\|)$ is a normed space. Let $X$ a compact subset of $E$ and let $N(X, \epsilon)$ denote the minimum number of balls of radius $\epsilon$ with centres in $X$ required to cover $X$. The (upper) box-counting dimension of $X$ is 

\begin{equation}\label{BC3}
\romd_{B}(X) = \limsup_{\epsilon \rightarrow 0} \frac{\log N(X, \epsilon)}{-\log \epsilon}.
\end{equation}

\end{definition}

It follows from the definition that if $d > \romd_{B}(X)$, then there exists some positive constant $C=C_{d}$, such that 

\begin{equation}
N(X, \epsilon) \leq C \epsilon^{-d}.
\end{equation}

We are interested in the notion of a weaker class of embeddings which are called almost bi--Lipschitz. Given $\delta \geq 0$, we say that a map $L \colon (X,d_{1}) \to (Y,d_{2})$, between two metric spaces is $\delta$-almost bi--Lipschitz if for $x\neq y \in X$, it satisfies
\[
	\frac{1}{C_{L}} \frac{d_{1}(x,y)}{\slog(d_{1}(x,y))^{\delta}} \leq d_{2}(L(x),L(y)) \leq C_{L} d_{1}(x,y),
\]
where $\slog(x) = \log\left(x + \frac{1}{x}\right)$ is defined as the symmetric logarithm function, for $x >0$.

In 2010, Olson \& Robinson \cite{OR} introduced a weaker notion of an $(\aaa,\bb)$--almost homogeneous metric space, which remains invariant under almost bi--Lipschitz maps.

\begin{definition}\label{homogeneousab}   
A metric space $(X,d)$ is $(\alpha, \beta)$-almost $(M,s)$- homogeneous if for any $0< \rho<r$
\begin{equation}\label{almosthomo}
	N_{X}(r,\rho) \leq M \left(\frac{r}{\rho}\right)^{s} \slog(r)^{\aaa} \slog(\rho)^{\bb}.
\end{equation}
\end{definition}

The authors studied the case when $X$ is a subset of a Hilbert space such that the difference $X-X$ is almost homogeneous and proved the following result.

\begin{theorem}\label{ETORASS}
Suppose that $H$ is a Hilbert space and $X \subset H$ such that $X-X$ is $(\aaa, \bb)$-almost homogeneous, for some $\aaa, \bb \geq 0.$ Then, for every 
\[
	\delta > \frac{3 + \aaa + \bb}{2},
\]
there exists a $N = N_{\delta} \in \mathbb{N}$ and a dense set of linear maps $L \colon H \to \mathbb{R}^{N}$ that are injective and $\delta$-almost bi-Lipschitz on $X$. 
\end{theorem}

\begin{remark}\label{ETORASSREMARK}
The condition on the set of differences on the result of Olson \& Robinson \cite{OR}, was used to control covers of balls around the origin. In particular, their result can be restated to hold for sets $X$ such that $X-X$ is almost homogeneous at zero.
\end{remark}

The set of embeddings in the above theorem is actually `prevalent', which is a notion of `almost every' in the context of infinite dimensional spaces. The authors used techniques introduced by Hunt \& Kaloshin \cite{HK} and also used by Margaris \& Robinson \cite{MargarisRob} to construct embeddings for sets with finite box--counting dimension. 

Sets of differences were again studied by Robinson \cite{Rob}, who proved that for any subset of a Banach space $X$ such that $X-X$ is homogeneous, $X$ admits $\delta$-almost bi--Lipschitz embeddings into Euclidean spaces, for all $\delta >1$. The exponent $\delta$ was shown to be sharp, in this case.

Unfortunately, the fact that $X$ is (almost) homogeneous does not necessarily imply that $X-X$ is also homogeneous (see examples of this kind of sets in chapter 9 in the book of Robinson \cite{JCR}).

We are interested in subsets of Banach spaces such that the set of differences is almost homogeneous at the origin, i.e. satisfy \eqref{almosthomo} only for balls around zero. The interesting fact about this condition is that it remains invariant under linear almost bi--Lipschitz maps, as we show in section \ref{ASSH}. In section \ref{ASSH}, we show that under this property we can construct a linear almost bi--Lipschitz embedding from a subset of a Banach space into a Hilbert space. Since the condition remains invariant, we then use Olson's \& Robinson's result (Theorem \ref{ETORASS}) to construct an embedding into an Euclidean space.

In section \ref{ASSEUC}, we extend the above result and we show that when $X-X$ is almost homogeneous at the origin, then there exists a prevalent set of linear almost bi--Lipschitz embeddings from $X$ into an Euclidean space. In particular, we show the following result.

\begin{theorem}\label{ET254}
Fix any $M \geq 1$, $s >0$ and $\aaa, \bb \geq 0$.  Suppose $X$ is a compact subset of a Banach space $\mathfrak{B}$ such that $X-X$ is $(\aaa, \bb)$-almost $(M,s)$-homogeneous at the origin. Then, given any $\delta >1+\frac{\aaa + \bb}{2}$, there exists a $N = N_{\delta} \in \mathbb{N}$ and a prevalent set of linear maps $L \colon \mathfrak{B} \to \mathbb{R}^{N}$ that are injective on $X$ and bi--Lipschitz with $\delta$--logarithmic corrections. In particular, they satisfy
\begin{equation}\label{WEM}
	\frac{1}{C_{L}} \frac{\|x-y\|}{\slog(\|x-y\|)^{\delta}} \leq |L(x) - L(y)| \leq C_{L} \|x-y\|,
\end{equation}
for some $C_{L} >0$ and for all $x,y \in X$.
\end{theorem}

We note that when $X-X$ is homogeneous, we obtain embeddings for any $\delta >1$, exactly as in Robinson's result. 

\section{Embedding into a Hilbert space when $X-X$ is almost homogeneous at the origin}\label{ASSH}

\begin{definition}\label{HOM0}
Suppose $M,s > 0$ and let $(X, |\cdot|)$ be a normed space. Then, given any $\aaa, \bb \geq 0$, we say that $X-X$ is $(\aaa, \bb)$-almost $(M, s)$- homogeneous at the origin if given any $0< \rho < r,$ there exist $z_{i} \in X-X$ such that
\[
	B_{r}(0) \cap X-X \subseteq \bigcup_{i=1}^{N} B_{\rho} (z_{i})
\]
and 
\[N \leq \left(\frac{r}{\rho}\right)^{s} \slog (r)^{\aaa} \slog (\rho)^{\bb}.\]
\end{definition}

We give some useful properties of the symmetric logarithm, which we will be using frequently. For the proof, see the paper of Olson \& Robinson \cite{OR}.

\begin{proposition}\label{properties of slog}
Let $C >0$ and $\gamma \geq 0$. There exist positive constants $A_{C}, B_{C}, a_{\gamma}, b_{\gamma}, c$ such that
\begin{enumerate}
\item $ |\log x| \leq \slog(x) \leq \log 2 + |\log x|$.
\item $A_{C} \, \slog(x) \leq \slog(C x) \leq B_{C} \, \slog(C x)$.
\item $ a_{\gamma} \, \slog(x) \leq \slog( x \slog(x)^{\gamma}) \leq b_{\gamma} \, \slog(x).$ 
\item  $\mathrm{If} \, \, \, \, 2^{-(k+1)} \leq x \leq 2^{-k}, \qquad \mathrm{then} \qquad \slog(x) \geq c \, \slog(2^{-k}).$
\end{enumerate}
\end{proposition}

It is trivial that when $X-X$ is $(\aaa, \bb)$-almost homogeneous, the above property is satisfied.
We now show that the condition on covers of balls around zero remains invariant for the sets of differences under linear almost bi--Lipschitz maps.

\begin{lemma}\label{Invariance}
Let $M \geq 1$, $s \geq 0$ and suppose $ \aaa, \bb \geq 0.$ Suppose $X$ is a compact subset of a Banach space $\mathfrak{B}$ and suppose that $X-X$ is $(\aaa, \bb))$-almost $(M, s)$- homogeneous at the origin. Let $\delta >1$ and let $\mathfrak{B'}$ be another Banach space. Suppose that $\Phi \colon \mathfrak{B} \to \mathfrak{B'}$ is a bounded linear map such that 
\begin{equation}\label{invariance for linear}
\frac{1}{C} \frac{\|x-y\|}{\slog(\|x-y\|)^{\delta}} \leq \|\Phi(x) - \Phi(y)\| \leq  C \|x-y\|,
\end{equation}
for some positive constant $C$ and for every $x,y \in X$.
Then, $\Phi(X)-\Phi(X) \subset \mathfrak{B'} $ is $(\aaa +\delta s, \bb)$-almost homogeneous at $0$.
\end{lemma}

\begin{proof}
Let $0<\rho <r.$ We want to cover the ball centred at $0$ in $\Phi(X) - \Phi(X)$. Let $x, y \in X$ be such that $\|\Phi(x) - \Phi(y)\| \leq r.$

Suppose first that $\|x - y\| \leq r$.
Then, since $X-X$ is almost homogeneous at $0$, there exist $z_{i} \in X-X$ such that
\[
	B_{r}(0) \cap X-X \subset \cup_{i=1}^{N} B_{\rho/C}(z_{i}),
\]
and $N \leq \left(\frac{r}{\rho}\right)^{s} \slog(r)^{\aaa} \slog(\rho)^{\bb}.$
Let $j \leq N$, such that $\|x- y-z_{j}\| \leq \rho/C$. Then, since $\Phi$ satisfies \eqref{invariance for linear}, we have
\[
	\|\Phi(x - y - z_{j})\| \leq C \|x-y-z_{j}\| \leq \rho.
\]
In particular we have that 
$$B_{r}(0) \cap (\Phi(X) - \Phi(X)) \subset \cup_{i=1}^{N} B_{\rho}(\Phi(z_{i}))$$ and 
$$N \leq \left(\frac{r}{\rho}\right)^{s} \slog(r)^{\aaa} \slog(\rho)^{\bb}.$$

Suppose now that $\|x - y\| > r$.
Then, let $R > 4$ be such that $X-X \subset B_{R/4}(0).$ Assume without loss of generality that $r < R/4.$
Then, by Lemma \ref{properties of slog}, we have
\begin{align*}
	\frac{\|x-y\|}{R} & \leq C \frac{1}{R} \|\Phi(x) - \Phi(y)\| \, \slog\left(\frac{\|x-y\|}{R}\right)^{\delta} \leq \frac{C r}{R} \log\left(\frac{R}{\|x-y\|}\right)^{\delta}\\
	& \leq \frac{C r}{R} \log\left(\frac{R}{r}\right)^{\delta} \leq \frac{C r}{R} \slog\left(\frac{r}{R}\right)^{\delta} \leq C_{R} \, r \slog(r)^{\delta}.
\end{align*}
Thus,
\[
	\|x - y\| \leq C_{R} \, r \slog(r)^{\delta},
\]
We now again use the fact that $X-X$ is almost homogeneous at $0$ to deduce a cover of the ball $B_{r \, \slog(r)^{\delta}}(0)$ in $X-X$ by at most $N$ balls of radius $\rho/C$. We now estimate $N$ using again properties of the symmetric logarithm (Lemma \ref{properties of slog}).
\begin{align*}
N & \leq \left(\frac{r \, \slog(r)^{\delta}}{\rho}\right)^{s} \slog(r \slog(r)^{\delta})^{\aaa} \slog(\rho)^{\bb} \leq b_{\delta} \left(\frac{r}{\rho}\right)^{s} \slog(r)^{\delta s} \slog(r)^{\aaa} \slog(\rho)^{\bb}\\
& \leq \left(\frac{r}{\rho}\right)^{s} \slog(r)^{\aaa + \delta s} \slog(\rho)^{\bb}.
\end{align*}
Arguing as in the previous case, $\{\Phi(z_{i})\}_{i=1}^{N} \subset \Phi(X)-\Phi(X)$ and we deduce that
\[
	B_{r}(0) \cap \Phi(X) - \Phi(X) \subset \bigcup_{i=1}^{N} B_{\rho}(\Phi(z_{i})),
\]
and 
\[
	N \leq \left(\frac{r}{\rho}\right)^{s} \slog(r)^{\aaa + \delta s} \slog(\rho)^{\bb}.
\]
In particular, $\Phi(X) - \Phi(X)$ is $(\aaa + \delta s, \bb)$- almost homogeneous.
\end{proof}

We now want to show that when $X$ is a subset of a Banach space such that $X-X$ is almost homogeneous at the origin, then it admits a linear almost bi--Lischitz embedding into a Hilbert space. In particular, by Lemma \ref{Invariance} the set of differences of the image of $X$ into $H$ will also satisfy the almost homogeneous property at $0$.
Following techniques used by Margaris \& Robinson \cite{MargarisRob}, we use a Hahn-Banach argument to construct the embedding at a single scale, as in the following Lemma.

\begin{lemma}\label{AssL}
Let $M \geq 1$, $R\geq 1$ and $s >0$ and suppose that $X$ is a compact subset of a Banach space $\mathfrak{B}$ such that $X-X$ is $(\aaa, \bb)$-almost $(M,s)$-homogeneous at $0$, for some $\aaa, \bb \geq 0$. Then, there exists a collection $\left(\phi_{n}\right)_{n=1}^{\infty}$ of elements of $\mathcal{L}\left(\mathfrak{B};\mathbb{R}^{m_{n}}\right)$ that satisfy $\|\phi_{n}\| \leq C_{R} \, \sqrt{m_{n}} \leq C_{R} \, n^{\frac{\aaa + \bb}{2}}$ and for every
$$x,y \in X \, \, \mbox{with} \, \, 2^{-(n+1)} R \leq \|x-y\| \leq R \, 2^{-n},$$ we have that $$|\phi_{n}(x-y)| \geq \frac{1}{4} \|x-y\|.$$
\end{lemma}

\begin{proof}
Suppose $Z = X-X$. Since $Z$ is $(\aaa, \bb))$-homogeneous we can cover 
$$Z\cap B_{ R 2^{-n}}(0) = \{z \in X-X : \|z\| \leq R \,2^{-n}\}$$ 
by no more than $$m_{n} \leq M \left(\frac{R \, 2^{-n}}{R \, 2^{-(n+3)}}\right)^{s} \slog(R \, 2^{-n})^{\aaa} \slog(R 2^{-(n+2)})^{\bb} \leq C n^{\aaa + \bb} \log 2 = C \, n^{\aaa + \bb} $$ balls of radius $R 2^{-(n+2)}$,
for some positive constant $C$ that depends only on $M, s, R$.
Let the centres of these balls be $z^{n}_{j}$, for $j \leq C \, n^{\aaa + \bb}$. 

Using the Hahn--Banach Theorem, we can find $f^{n}_{j} \in \mathfrak{B}^{*}$ such that $\|f^{n}_{j}\| = 1$ and $f^{n}_{j}(z^{n}_{j}) = \|z^{n}_{j}\|.$
Now, define $\phi_{n} \colon \mathfrak{B} \to \mathbb{R}^{m}$ by
$$\phi_{n}(x) = \left(f^{n}_{1}(x),...,f^{n}_{m_{n}}(x)\right).$$
It is clear that $\|\phi_{n}\| \leq \sqrt{m_{n}} \leq C \, n^{\aaa+\bb/2}$ and if $z = x-y \in X-X$ is such that $$2^{-(n+1)} R \leq \|z\| \leq R \,2^{-n},$$ then for some $z^{n}_{j}$ we have that $\|z-z^{n}_{j}\| \leq R 2^{-(n+2)}$ and so
\begin{align*}
|\phi_{n}| & \geq |f^{n}_{j}(z)| = |f^{n}_{j}(z-z^{n}_{j}+z^{n}_{j})|\\
& \geq |f^{n}_{j}(z^{n}_{j})| - |f^{n}_{j}(z-z^{n}_{j})| \geq \|z^{n}_{j}\| - \|z-z^{n}_{j}\|\\
& \geq \|z\| - 2\|z-z^{n}_{j}\| \geq R 2^{-(n+1)} - R 2^{-(n+2)} \geq \frac{1}{4}\|z\|,
\end{align*}
which concludes the proof of the embedding at a single scale.
\end{proof}

The above Lemma allows for the following embedding.

\begin{theorem}\label{Almost homogeneous H}
Suppose that $X$ is a compact subset of a Banach space $\mathfrak{B}$ such that $X-X$ is $(\aaa, \bb)$-almost homogeneous at the origin. Then, given any $\delta >\frac{1 + \aaa + \bb}{2}$, there exists a Hilbert space $H$ and a bounded linear map $\Phi \colon \mathfrak{B} \to H$, that satisfies
$$\frac{1}{C_{\Phi}} \frac{\|x-y\|}{\slog(\|x-y\|)^{\delta}} \leq \|\Phi(x) - \Phi(y)\| \leq \|x-y\|,$$
for some positive constant $C_{\Phi}$ and for every $x,y \in X$.
\end{theorem}

\begin{proof}
Take $R >6$ such that 
\[
	X-X \subset B_{R/2}(0) \subset B_{R}(0).
\]
Take $\delta$ such that
$$\delta > \frac{1 + \aaa + \bb}{2}.$$
Let $m_{n}, \phi_{n}$ be from Lemma \eqref{AssL}. Suppose $\{{e_{k}}\}_{k=1}^{m}$ is a basis for $\mathbb{R}^{m}$, which we cyclically extend to all $k \in \mathbb{N}$, as in the previous chapter.
Then, we define $\Phi \colon \mathfrak{B} \to H$ by 
$$\Phi(x) = \sum_{k=1}^{\infty} k^{-\delta} \phi_{k}(x) \otimes \hat{e_{k}}.$$
Then, $\Phi$ is obviously linear and for every $x \in \mathfrak{B}$, we have
$$\|\Phi(x)\|^{2} \leq \sum_{n=1}^{\infty} |n^{- \delta} \phi_{n}(x)|^{2} \leq \|x\|^{2} \sum_{n=1}^{\infty} n^{- 2 \delta} n^{\aaa + \bb }  = \|x\|^{2} \sum_{n=1}^{\infty} n^{- 2\delta + \aaa + \bb} < \infty.$$
Hence 
\[
	\|\Phi\| \leq \sum_{n=1}^{\infty} n^{- 2\delta + \aaa + \bb} < \infty,
\]
since $\aaa + \bb - 2\delta < -1.$
Then, for any $x , y \in X$, let $k \geq 1$ be such that
\[
	2^{-(k+1)} R \leq \|x- y\| \leq R 2^{-k}.
\] 
By definition of $\Phi$, we have
\begin{align*}
\|\Phi(x) - \Phi(y)\| = \|\Phi(x-y)\| \geq k^{-\delta} |\phi_{k}(x-y)| \geq k^{-\delta} \frac{1}{4} \|x-y\|.
\end{align*}
Using properties of the symmetric logarithm (Lemma \ref{properties of slog}), we obtain
\begin{align*}
\slog(\|x-y\|)^{\delta} & \geq A_{R} \, \slog\left(\frac{\|x-y\|}{R}\right)^{\delta} \geq A_{R} \, b \, \slog(2^{-k})^{\delta}\\
& \geq A_{R} \, b \, k^{\delta},
\end{align*}
for constants $A_{R}, b$ independent of $x,y$. Thus, 
\[
	\|\Phi(x) - \Phi(y)\| \geq \frac{1}{C_{\Phi}} \frac{\|x-y\|}{\slog(\|x-y\|)^{\delta}}.\qedhere
\]
\end{proof}

Hence, by Theorem \ref{ETORASS} and Remark \ref{ETORASSREMARK}, we immediately obtain an almost bi--Lipschitz embedding into an Euclidean space. However, in the section we establish the existence of a prevalent set of almost bi--Lipschitz maps into Euclidean spaces directly.

We also note that the above theorem can be used to provide embeddings of compact metric spaces, using the isometric embedding $\Phi^{*} \colon (X,d) \to L^{\infty}(X)$, given by $x \mapsto d(x, \cdot),$ due to Kuratowski, which was also mentioned in the previous chapter (see Lemma \ref{ET100}). In particular, we can define $`X-X'$ in this context to mean
$$X-X \equiv \Phi^{*}(X)-\Phi^{*}(X) =  \{f \in L^{\infty}(X) : f = d(x,\cdot) - d(y,\cdot), \, \, \text{for} \, \, x, y \, \, \text{in} \, \, X\}.$$

\section{Embeddings into an Euclidean space}\label{embedding almost homogeneous $0$}\label{ASSEUC}

\subsection{A measure based on sequences of linear subspaces}\label{M}

Before we prove our main embedding result, we will recall, following Robinson \cite{JCR}, the construction of a compactly supported probability measure that is based on the ideas in Hunt and Kaloshin \cite{HK} and will play a key role in our proof.

Suppose that $\mathfrak{B}$ is a Banach space and $\mathcal{V} = \{V_{n}\}_{n=1}^{\infty}$ a sequence of finite--dimensional subspaces of $\mathfrak{B}^{*}$, the dual of $\mathfrak{B}$.
Let us denote by $d_{n}$ the dimension of $V_{n}$ and by $B_{n}$ the unit ball in $V_{n}$. 

Now, we fix a real number $\alpha >1$ and define the space $\mathbb{E}_{\alpha}(\mathcal{V})$ as the collection of linear maps $L \colon \mathfrak{B} \to \mathbb{R}^{k}$ given by 
\[\mathbb{E} = \mathbb{E}_{\alpha}(\mathcal{V}) = \left\lbrace L = (L_{1}, L_{2},..., L_{k}) : L_{i} = \sum_{n=1}^{\infty} n^{-\alpha} \phi_{i,n} , \, \, \phi_{i,n} \in B_{n}\right\rbrace.\]
Let us also define
\[\mathbb{E}_{0} = \left\lbrace \sum_{n=1}^{\infty} n^{-\alpha} \phi_{i,n} , \, \, \phi_{i,n} \in B_{n}\right\rbrace.\]
Clearly $\mathbb{E} = \left(\mathbb{E}_{0}\right)^{k}$.

To define a measure on $\mathbb{E}$, we first take a basis for $V_{n}$ so that we can identify $B_{n}$ with a symmetric convex set $U_{n} \subset \mathbb{R}^{d_{n}}$. Then, we construct each $L_{i}$ randomly by choosing each $\phi_{i,n}$ with respect to the normalised $d_{n}$--dimensional Lebesgue measure $\lambda_{n}$ on $U_{n}$. Finally, by taking $k$ copies of this measure we obtain a measure on $\mathbb{E}.$
In particular we first consider $\mathbb{E}_{0}$ as a product space
\[\mathbb{E}_{0} = \prod_{n=1}^{\infty} B_{n},\] and define a measure $\mu_{0}$ on $\mathbb{E}_{0}$ as
\[\mu_{0} = \otimes_{n=1}^{\infty} \lambda_{n}.\]
Secondly, we consider  $\mathbb{E} = \mathbb{E}_{0}^{k}$ and define $\mu$ on $\mathbb{E}$ as
\[\mu = \prod_{i=1}^{k} \mu_{0}.\]

For any map $f \in \mathcal{L}(\mathfrak{B};\mathbb{R}^{k})$, Hunt and Kaloshin \cite{HK} proved the following upper bound on

\[\mu\{ L \in \mathbb{E} : |(f+L)x| \leq \epsilon\},\]
for $x \in \mathcal{B}$ and any $\epsilon > 0.$ For a more detailed proof, see Robinson \cite{JCR}.

\begin{lemma}\label{1.6}
Suppose that $x \in \mathcal{B}$, $\epsilon > 0$, $f \in \mathcal{L}(\mathfrak{B};\mathbb{R}^{k})$ and $\mathcal{V}=\{V_{n}\}$ as above. Then
\[\mu\{L \in \mathbb{E} : |(f + L)(x) | < \epsilon \} \leq  \left(n^{\alpha}d_{n} \frac{\epsilon}{|g(x)|}\right)^{k},\]
for any $g \in B_{n}$.
\end{lemma}

\subsection{Prevalent set of embeddings when $X-X$ is almost homogeneous at the origin}

We now extend the result of the previous section and prove the existence of a prevalent set of almost bi--Lipschitz embeddings into Euclidean spaces for a compact subset of a Banach space such that $X-X$ is almost homogeneous at the origin. 

We first show that if $X-X$ is almost homogeneous at zero then the box--counting dimension of $X-X$ is finite.

\begin{lemma}\label{DBAA}
Take any $M \geq 1$ and $s \geq 0$. Suppose that $X$ is a compact subset of a Banach space $\mathfrak{B}$ such that $X-X$ is $(\aaa,\bb)$-almost homogeneous at the origin. Then,
$\romd_{B}(X-X) < \infty.$
\end{lemma}

\begin{proof}
Take any $\epsilon >0.$ Let $R > 1$ be such that
$$X-X \subset B_{R}(0).$$

Suppose that $1 < \epsilon < R$.
Since $X-X$ is almost homogeneous at $0$, there exist $\{z_{i}\}_{i=1}^{N} \subset X-X$ such that
\[
	X-X \subset B_{R}(0) \cap X-X \subseteq \cup_{i=1}^{N} B_{\epsilon}(z_{i}).
\]
Using properties of the $\slog$ function, we have 
\begin{align*}
N \leq \left(\frac{R}{{\epsilon}}\right)^{s} \slog (R)^{\aaa} \slog (\epsilon)^{\bb} \leq R^{s} (\slog R)^{\aaa} (\log 2 + \log \epsilon)^{\bb} \epsilon^{-s} \leq C_{R} \,  \epsilon^{-s}.
\end{align*}
Suppose now that $\epsilon < 1 < R$.
Since $X-X$ is almost homogeneous at $0$, there exist $\{z_{i}\}_{i=1}^{N} \subset X-X$ such that
\[
	X-X \subset B_{R}(0) \cap X-X \subseteq \cup_{i=1}^{N} B_{\epsilon}(z_{i}),
\]
and 
\begin{align*}
N & \leq \left(\frac{R}{{\epsilon}}\right)^{s} \slog (R)^{\aaa} \slog (\epsilon)^{\bb} \leq R^{s} (\slog R)^{\aaa} (\log \epsilon + \log \frac{1}{\epsilon})^{\bb} \epsilon^{-s} \\
& \leq R^{s} (\slog R)^{\aaa} \left(\log \frac{1}{\epsilon}\right)^{\beta} \epsilon^{-s} \leq R^{s} (\slog R)^{\aaa} \epsilon^{-\beta-s},
\end{align*}
which immediately implies that
\[
\romd_{B}(X-X) \leq \bb + s < \infty.\qedhere
\]
\end{proof} 

We now recall the result due to Robinson \cite{Rob} which provides a prevalent set of injective and bi--H\"older embeddings from $X$ into an Euclidean space when $\romd_{B}(X-X) < \infty$.

\begin{theorem}\label{ET154}
Suppose $\mathfrak{B}$ is a Banach space and let $X \subset \mathfrak{B}$ compact such that $\romd_{B}(X-X) < \infty$. Then for any integer $k > \romd_{B}(X-X)$ and any given $\theta$ with
\[0 < \theta < \frac{k -\romd_{B}(X-X)}{k\left(1 + \romd_{B}(X-X)\right)},\]
there exist a prevalent set of linear maps $L \colon \mathfrak{B} \to \mathbb{R}^{k}$ such that

$$
\|x-y\| \leq C_{L}|Lx-Ly|^{\theta}, \, \, \, \forall \, \, \, x,y \in X, \,\, \mbox{for some}\,\, C_{L}>0.
$$
In particular, every such $L$ is bijective from $X$ onto $L(X)$ with a H\"older continuous inverse.
\end{theorem}

We are now in position to state and prove the main result of this paper. For the proof, we use techniques introduced by Olson \& Robinson \cite{OR} and also used by Robinson \cite{Rob}, tailored to the weaker condition we now have at the origin.

\begin{theorem}\label{ET254}
Fix any $M \geq 1$, $s >0$ and $\aaa, \bb \geq 0$.  Suppose $X$ is a compact subset of a Banach space $\mathfrak{B}$ such that $X-X$ is $(\aaa, \bb)$-almost $(M,s)$-homogeneous at the origin. Then, given any $\delta >1+\frac{\aaa + \bb}{2}$, there exists a $N = N_{\delta} \in \mathbb{N}$ and a prevalent of linear maps $L \colon \mathfrak{B} \to \mathbb{R}^{N}$ that are injective on $X$ and bi--Lipschitz with $\delta$--logarithmic corrections. In particular, they satisfy
\begin{equation}\label{WEM}
	\frac{1}{C_{L}} \frac{\|x-y\|}{\slog(\|x-y\|)^{\delta}} \leq |L(x) - L(y)| \leq C_{L} \|x-y\|,
\end{equation}
for some $C_{L} >0$ and for all $x,y \in X$.
\end{theorem}

\begin{proof}
The proof consists of three parts. We first establish the existence of a prevalent set $T_{1}$ of linear maps $L$ that satisfy \eqref{WEM}, for all $x,y \in X$ such that $\|x - y\| \leq r_{L},$ for some $r_{L} > 0$. We then use Theorem \ref{ET154} to construct a prevalent set $T_{2}$ of linear maps that are injective on $X$ and have a H\"older continuous inverse. Finally, we show that all linear maps in $T_{1} \cap T_{2}$, which in particular is a prevalent set, satisfy \eqref{WEM}, for all $x,y \in X$.

Let $Z = X-X.$ Let $R > 6$ be such that
\[
	Z \subset B_{R/2}(0) \subset B_{R}(0).
\]
Let $\gamma >1$ be such that
$$\delta > \frac{\aaa + \bb}{2} + \gamma > \frac{\aaa + \bb}{2} +1.$$
By Lemma \ref{AssL}, for any given $n \in \mathbb{N}$, there exist a collection of functionals $\{f^{n}_{i}\}_{i=1}^{m_{n}} \subset \mathfrak{B}^{*}$ with $m_{n} \leq C n^{\aaa + \bb/2}$, $\|f^{n}_{i}\| = 1$ and such that for any $z \in Z$ that satisfies $ R\,  2^{-n+1} \leq \|z\| \leq R \, 2^{-n}$, there exists $f^{n}_{j}$ such that 
$$|f^{n}_{j}(z)| \geq 2^{(-n+3)}.$$
Let 
$$V_{n} = \lspan\{f^{n}_{1},...,f^{n}_{m_{n}}\}.$$
Let also $N \in \mathbb{N}$. Based on the sequence $\mathcal{V} = \{V_{n}\}_{n=1}^{\infty}$ and on $\gamma >1$, we follow the construction in the previous chapter and we define a probe space $\mathbb{E_{\gamma}} = \mathbb{E_{\gamma}}(\mathcal{V}) \subset \mathcal{L}\left(\mathcal{B}, \mathbb{R}^{N}\right)$ with a measure $\mu$ compactly supported on $\mathbb{E_{\gamma}}(\mathcal{V})$.

Following the argument of the previous chapter, we fix a map $f \in \mathcal{L}(\mathfrak{B} ; \mathbb{R}^{N})$ and suppose that $K'$ is a Lipschitz constant that holds for all $L \in \mathbb{E}.$ Then, we define
\[
	Z_{n} = \{z \in Z: 2^{(-n-1)}\, R \leq \|z\| \leq 2^{-n}\, R\}
\] and
\[
	Q_{n} = \{L \in \mathbb{E}: |(f+L)(z)| \leq n^{-\delta} 2^{-n}, \qquad \text{for some} \qquad z \in Z_{n}\}.
\] 

Since $X$ is $(\aaa, \bb)$- almost homogeneous at $0$, given any $n \in \mathbb{N}$, there exist $\{z^{n}_{i}\}_{i=1}^{k_{n}} \subset Z$ such that
\[
	Z_{n} \subset B_{R\, 2^{-n}}(0) \cap Z \subset \bigcup_{i=1}^{k_{n}} B_{n^{-\delta} 2^{-n}}(z_{i}),
\]
and
\[
 k_{n} \leq M \left(\frac{R \,2^{-n}}{n^{-\delta} 2^{-n}}\right)^{s} \slog(R\, 2^{-n})^{\aaa} \slog(n^{-\delta}2^{-n})^{\bb}\leq C\left(n^{\delta}\right)^{s} n^{\aaa + \bb},
\]
for some positive constant $C$ depending on $\aaa, \bb, M$.
Now let $L \in Q_{n}$. Then there exists $z \in Z_{n}$ such that $|(f+L)(z)| \leq n^{- \delta}2^{-n}.$
Since $z \in Z_{n}$, there exists $z^{n}_{i}$ such that
\[
	\|z - z^{n}_{i}\| \leq n^{-\delta} 2^{-n},
\]
which implies that
\begin{align*}
	|(f+L)(z^{n}_{i})| & \leq |(f+L)(z^{n}_{i}) + (f+L)(z) - (f+L)(z)|\\
	& \leq n^{-\delta} 2^{-n} + (\|f\|+\|L\|)n^{-\delta}2^{-n}\\& \leq (1 + \|f\| + K') n^{-\delta}2^{-n} = K n^{-\delta}2^{-n},
\end{align*}
where $K$ depends on $f$.

We now compute the measure of $Q_{n}$, based on Lemma \ref{1.6}. In particular, we have
\begin{align*}
\mu(Q_{n}) & \leq \sum_{i=1}^{k_{n}} \mu \{L \in \mathbb{E}: |(f+L)(z^{n}_{i})| \leq (1+K) n^{-\delta} 2^{-n}\}\\
& \leq k_{n} \left( \dim (V_{n}) \, n^{\gamma} K n^{-\delta} 2^{-n} |\phi(z^{n}_{i}|^{-1}\right)^{N},
\end{align*}
for any $\phi \in B_{n}$, the unit ball in $V_{n}$
Since $z^{n}_{i} \in Z_{n}$, there exists $f^{n}_{i} \in V_{n}$ such that $\|f^{n}_{i}\| = 1$ and $|f^{n}_{i}(z^{n}_{i})| \geq 2^{-(n+3)}.$
Therefore,
\[
	\mu(Q_{n}) \leq C n^{\delta s + \aaa + \bb  + \frac{\aaa + \bb}{2} N + \gamma N - \delta N}.
\]

Since $\delta > \frac{\aaa + \bb}{2} + \gamma$, we can choose $N$ big enough such that 
$$\frac{\left(\frac{\aaa + \bb}{2}+\gamma \right) N +1}{N - s} < \delta,$$
which implies that
$$\delta s + \frac{\aaa + \bb}{2}N + \gamma N - \delta N< -1.$$

Therefore, $\sum_{n=1}^{\infty}Q_{n} < \infty$ and by the Borel-Cantelli Lemma, for $\mu$-almost every $L \in \mathbb{E}$, there exists an $n_{L} \geq 1$ such that for all $n \geq n_{L}$
\[
	2^{-(n+1)} \, R \leq |z| \leq 2^{-n}\, R \qquad \Rightarrow \qquad |(f+L)z| \geq n^{-\delta}2^{-n}.
\]
Let $z \in Z$. If $\|z\| \leq R 2^{-n_{L}}$, then there exists $n \geq n_{L} \geq 1$ such that
\[
	2^{-(n+1)} R \leq \|z\| \leq 2^{-n} \, R.
\]
Therefore, arguing as in the end of Theorem \ref{Almost homogeneous H}, we obtain
\begin{align*}
|(f+L)(z)| \geq n^{-\delta} 2^{-n} \geq n^{-\delta} 2^{-n} \geq \frac{1}{A_{R}} \frac{\|z\|}{\slog(\|x-y\|)^{\delta}}.
\end{align*}
Thus, we proved that there exists a prevalent set of bounded linear maps $L \colon \mathbb{B} \to \mathbb{R}^{N}$, denoted by $T_{1}$ such that all $L \in T_{1}$ 
\[
 \frac{1}{C'} \frac{\|x-y\|}{\slog(\|x-y\|)^{\delta}} \leq |L(x) - L(y)| \leq C' \|x-y\|,
 \]
for all $x,y \in X$ such that $\|x - y\| \leq R 2^{-n_{L}}$, for some $n_{L} \geq 1.$ By Lemma \ref{DBAA}, we know that $\romd_{B}(X-X)<\infty$. Hence, by Theorem \ref{ET154}, for a fixed $\theta <1$,  we establish the existence of a $N_{1} \in \mathbb{N}$ and another prevalent set of linear maps $L \colon \mathbb{B} \to \mathbb{R}^{N_{1}}$, denoted by $T_{2}$ such that any $L \in T_{2}$ is $\theta$-bi-H\"older on $X$. Assume without loss of generality that $N_{1}\leq N$ and let $T = T_{1} \cap T_{2}$, which is still prevalent. Now suppose that $L \in T$ and $f \in \mathcal{L}(\mathfrak{B}, \mathbb{R}^{N})$.

Let $z \in Z$. Let $m \geq 1$ such that
\[
	2^{-(m+1)} R \leq \|z\| \leq 2^{-m} R.
\]
If $m \leq n_{L}$, we use that $f+L \in T \subset T_{1}$ and we fall in the previous case that we just proved.
Suppose now that $m > n_{L} \geq 1$. Then, $f+L \in T \subset T_{2}$. In particular,
\[
	|(f+L)(z)| \geq \|z\|^{1/\theta} \geq R^{1/\theta} 2^{-n_{L} /\theta}. 
\]
Hence,
\begin{align*}
\slog(\|x-y\|)^{\delta} & \geq C_{R} \slog\left(\frac{\|x-y\|}{R}\right)^{\delta} \\
& \geq C_{R}\, |\log 2^{-m}|^{\delta} =  C_{R}\,  m^{\delta} \log 2 >  A_{R}\,  n_{L}^{\delta} \, \log 2.
\end{align*} 
Thus,
\begin{align*}
|(f+L)(z)| \geq R^{1/\theta} 2^{-n_{L} /\theta}   A_{R} \,  n_{L}^{\delta} \, \log 2 \frac{\|x-y\|}{\slog(\|x-y\|)^{\delta}},
\end{align*}
and the proof is now complete.
\end{proof}

Arguing as in the previous case, we can extend the above theorem for any compact metric space, using the Kuratowski embedding. In particular, the following theorem holds.

\begin{theorem}
Suppose that $(X,d)$ is a metric space and let $\Phi \colon X \to L^{\infty}(X)$ be the Kuratowski embedding. Suppose that $\Phi(X)-\Phi(X)$ is $(\aaa, \bb)$-almost homogeneous at $0$ Then, for any given $\delta > (\aaa + \bb)/2 + 1$, there exists a $N = N_{\delta} \in \mathbb{N}$ and a map $L \colon (X,d) \to \mathbb{R}^{N}$, which is bi--Lipschitz with $\delta$-logarithmic corrections and injective.
\end{theorem}

\section{Conclusion and Future Work}

We proved an embedding theorem for subsets of Banach spaces such that the set of differences is almost homogeneous at $0$. The theorem is an extension of the respective result for subsets of Hilbert spaces, such that $X-X$ is almost homogeneous, that was proved by Olson \& Robinson \cite{OR}. In particular, we prove the theorem under a weaker condition which only deals with balls around $0$ rather than any point in $X-X$. There are a number of open questions that arise naturally from the results we presented and we would like to list some of them here.

\begin{Q}\label{question dsexte}
In Section \ref{embedding almost homogeneous $0$}, we construct a prevalent set of linear almost bi--Lipschitz embeddings from a subset of Banach space $X$ into some Euclidean space, when $X-X$ is almost homogeneous at $0$. Is it possible to extend the theorem for doubling subsets of Banach spaces? 
\end{Q}

\begin{Q}
Suppose $X$ is a doubling subset of a Banach space. Can we find an almost bi--Lipschitz embedding $\psi$ from $X$ into another Banach space $Y$ such that the set of differences $\psi(X) -  \psi(X)$ into $Y$ is almost homogeneous at $0$? This would yield a positive answer to the question above.
\end{Q}

\begin{Q}
In Section \ref{ASSH}, we show that if $X-X$ is almost homogeneous and can be embedded into a Hilbert space, using a linear almost bi--Lipschitz map $\Phi$, then $\Phi(X) - \Phi(X)$ is almost homogeneous at $0$. This leads to the following question. Is the condition that  $X-X$ is almost homogeneous at $0$ necessary in order to have linear bi--Lipschitz embeddings into some Euclidean space? (we already know by Theorem \ref{ET254} that it is sufficient).
\end{Q}

\bibliographystyle{plain}
\bibliography{ReferencesLipschitz}
\end{document}